\documentclass[12pt]{amsart}

\usepackage{amssymb,amsmath,latexsym,enumerate,graphicx,bbm,mathptmx,microtype,cite} 

\makeatletter
\@namedef{subjclassname@2010}{%
  \textup{2010} Mathematics Subject Classification}
\makeatother

\newtheorem{thm}{Theorem}[section]
\newtheorem{cor}[thm]{Corollary}
\newtheorem{lem}[thm]{Lemma}
\newtheorem{prop}[thm]{Proposition}

\theoremstyle{definition}

\numberwithin{equation}{section}

\renewcommand\th{^{\text{th}}}

\newcommand\commentout[1]{}

\newcommand\ZZ{\mathbb{Z}}
\newcommand\QQ{\mathbb{Q}}
\newcommand\RR{\mathbb{R}}
\newcommand\CC{\mathbb{C}}

\newcommand{\C}{\mathbb{C}}

\def\res{\mathop{\text{Res}}}

\begin{document}


\baselineskip=17pt



\title{Reciprocity Theorems for Bettin--Conrey Sums}

\author{Juan S. Auli}
\address{Department of Mathematics\\
         Dartmouth College\\
         Hanover, NH 03755\\
         U.S.A.}
\email{juan.s.auli.gr@dartmouth.edu}

\author{Abdelmejid Bayad}
\address{Abdelmejid Bayad\\ 
        Universit\'e d'\'Evry Val d'Essonne\\
        Laboratoire de Math\'ematiques et Mod\'elisation d'\'Evry (CNRS-UMR 8071),
I.B.G.B.I., 23 Bd. de France, 91037 \'Evry Cedex, France}
        \email{abayad@maths.univ-evry.fr}
\author{Matthias Beck}
\address{Department of Mathematics\\
         San Francisco State University\\
         San Francisco, CA 94132\\
         U.S.A.}
\email{mattbeck@sfsu.edu}

\dedicatory{Dedicated to the memory of Tom M.\ Apostol}

\date{20 January 2017}

\begin{abstract}
Recent work of Bettin and Conrey on the period functions of Eisenstein series 
naturally gave rise to the Dedekind-like sum
\[
  c_{a}\left(\frac{h}{k}\right) \ = \ 
  k^{a}\sum_{m=1}^{k-1}\cot\left(\frac{\pi mh}{k}\right)\zeta\left(-a,\frac{m}{k}\right),
\]
where $a \in \CC$, $h$ and $k$ are positive coprime integers, and $\zeta(a,x)$ denotes the Hurwitz zeta function.
We derive a new reciprocity theorem for these \emph{Bettin--Conrey sums}, which in the case of an odd negative integer $a$ can be explicitly given in terms of Bernoulli
numbers. This, in turn, implies explicit formulas for the period functions appearing in Bettin--Conrey's work.
We study generalizations of Bettin--Conrey sums involving zeta derivatives and multiple cotangent factors and relate these to special values of the Estermann zeta function. 
\end{abstract}



\maketitle

\section{Introduction and Statement of Results}

Our point of departure is recent work of Bettin and Conrey
\cite{bettinconreyreciprocity,bettinconreyperiodfunctions} on the period functions of Eisenstein
series. Their initial motivation was the derivation of an exact formula for the second moments of the Riemann zeta function, but their work naturally gave rise to a family of
finite arithmetic sums of the form
\[
  c_{a}\left(\frac{h}{k}\right) \ = \ 
  k^{a}\sum_{m=1}^{k-1}\cot\left(\frac{\pi mh}{k}\right)\zeta\left(-a,\frac{m}{k}\right),
\]  
where $a \in \CC$, $h$ and $k$ are positive coprime integers, and $\zeta(a,x)$ denotes the \emph{Hurwitz zeta
function}
\[
  \zeta(a,x) \ = \
  \sum_{n=0}^{\infty}\frac{1}{(n+x)^{a}} \, ,
\]
initially defined for $\Re(a)>1$ and meromorphically continued to the $a$-plane.
We call $c_{a}(\frac{h}{k})$ and its natural generalizations appearing below \emph{Bettin--Conrey sums}.

There are two major motivations to study these sums.
The first is that $c_{0}(\frac{h}{k})$ is essentially a \emph{Vasyunin sum}, which in turn makes
a critical appearance in the Nyman--Beurling--B\'aez-Duarte approach to the Riemann hypothesis
through the twisted mean-square of the Riemann zeta function on the critical line (see, e.g.,
\cite{baez,vasyunin}). Bettin--Conrey's work, for $a=0$, implies that there is a hidden symmetry of this mean-square.

The second motivation, and the central theme of our paper, is that the Bettin--Conrey sums
satisfy a \emph{reciprocity theorem}:
\[
  c_{a}\left(\frac{h}{k}\right)-\left(\frac{k}{h}\right)^{1+a}c_{a}\left(\frac{-k}{h}\right)+\frac{k^{a}a \, \zeta(1-a)}{\pi h}
\]
extends from its initiation domain $\QQ$ to an (explicit) analytic function on $\CC \setminus \RR_{
\le 0 }$, making $c_a$ nearly an example of a \emph{quantum modular form} in the sense of Zagier
\cite{zagierquantummodular}. In fact, Zagier's ``Example 0'' is the \emph{Dedekind sum}
\[
  s(h,k) \ = \
  \frac{1}{4k}\sum_{m=1}^{k-1}\cot\left(\frac{\pi mh}{k}\right)\cot\left(\frac{\pi m}{k}\right) ,
\]
which is, up to a trivial factor, $c_{ -1 }(\frac h k)$. Dedekind sums first appeared in the
transformation properties of the Dedekind eta function and satisfy the reciprocity theorem
\cite{dedekind,grosswald}
\[
  s(h,k) + s(k,h) \ = \
  - \frac 1 4 + \frac 1 {12} \left( \frac h k + \frac 1 {hk} + \frac k h \right) .
\] 
We now recall the precise form of Bettin--Conrey's reciprocity theorem.

\begin{thm}[Bettin--Conrey \cite{bettinconreyperiodfunctions}]\label{thm:bettinconrey}
If $h$ and $k$ are positive coprime integers then
\[
  c_{a}\left(\frac{h}{k}\right)-\left(\frac{k}{h}\right)^{1+a}c_{a}\left(\frac{-k}{h}\right)+\frac{k^{a}a \, \zeta(1-a)}{\pi h} \ = \
  -i \, \zeta(-a)\, \psi_{a}\left(\frac{h}{k}\right)
\]
where
\[
  \psi_{a}(z) \ = \
   \frac{i}{\pi z}\frac{\zeta(1-a)}{\zeta(-a)}-\frac{i}{z^{1+a}}\cot\frac{\pi a}{2}+i\frac{g_{a}(z)}{\zeta(-a)}
\]
and
\begin{align*}
  g_{a}(z) \ = \ &-2\sum_{1\leq n\leq M}(-1)^{n}\frac{B_{2n}}{(2n)!}\, \zeta(1-2n-a)(2\pi z)^{2n-1} \\
  &\qquad {}+\frac{1}{\pi i}\int_{(-\frac{1}{2}-2M)}\zeta(s) \, \zeta(s-a) \,
\Gamma(s)\frac{\cos\frac{\pi a}{2}}{\sin\pi \frac{s-a} 2}(2\pi z)^{-s} \, ds \, .
\end{align*}
Here $B_k$ denotes the $k\th$ Bernoulli number, $M$ is any integer $\ge -\frac{1}{2}\min(0,\Re(a))$, and the integral notation indicates that our integration path is over the
vertical line $\Re(s) = -\frac{1}{2}-2M$.
\end{thm}

We note that Bettin and Conrey initially defined $\psi_a(z)$ through
\[
  \psi_a(z) \ = \
  E_{ a+1 } (z) - \frac{ 1 }{ z^{ a+1 } } \, E_{ a+1 } \left( - \frac 1 z \right) ,
\]
in other words, $\psi_a(z)$ is the \emph{period function} of the \emph{Eisenstein series} of weight
$a+1$,
\[
  E_{ a+1 } (z) \ = \
  1 + \frac{ 2 }{ \zeta(-a) } \sum_{ n \ge 1 } \sigma_a(n) \, e^{ 2 \pi i n z } ,
\]
where $\sigma_a(n) = \sum_{ d|n } d^a$,
and then showed that $\psi_a(z)$ satisfies the properties of Theorem~\ref{thm:bettinconrey}.

We have several goals.
We start by showing that the right-hand side of Theorem~\ref{thm:bettinconrey} can be simplified
by employing an integration technique for Dedekind-like sums that goes back to Rademacher
\cite{grosswald}. This yields our first main result:

\begin{thm}\label{thm:GeneralReciprocity}
Let $\Re(a)>1$ and suppose $h$ and $k$ are positive coprime integers.
Then for any $0<\epsilon<\min\left\{ \frac{1}{h},\frac{1}{k}\right\}$,
\[
  h^{1-a} \, c_{-a}\left(\frac{h}{k}\right)+k^{1-a} \, c_{-a}\left(\frac{k}{h}\right) \ = \
  \frac{a \, \zeta(a+1)}{\pi(hk)^{a}}-\frac{(hk)^{1-a}}{2i}\int_{(\epsilon)}\frac{\cot(\pi hz)\cot(\pi kz)}{z^{a}} \,
dz \, .
\]
\end{thm}

Theorem~\ref{thm:GeneralReciprocity} implies 
that the function 
\[
  F(a) \ = \ \int_{(\epsilon)}\frac{\cot(\pi hz)\cot(\pi kz)}{z^{a}} \, dz
\]
has a holomorphic continuation to the whole complex plane.
In particular, in this sense Theorem~\ref{thm:GeneralReciprocity} can be extended to all complex~$a$. 

Second, we employ Theorem~\ref{thm:GeneralReciprocity} to show that in the case that $a$ is an
odd negative integer, the right-hand side of the reciprocity theorem can be explicitly given in terms of Bernoulli numbers.

\begin{thm}\label{thm:nReciprocity}
Let $n>1$ be an odd integer and suppose $h$ and $k$ are positive coprime integers. Then
\begin{align*}
  &h^{1-n} \, c_{-n}\left(\frac{h}{k}\right)+k^{1-n} \, c_{-n}\left(\frac{k}{h}\right) \ = \\
  &\qquad \left(\frac{2\pi i}{hk}\right)^{n}\frac{1}{i(n+1)!}\left(n \, B_{n+1}+\sum_{m=0}^{n+1}{n+1
\choose m}B_{m} \, B_{n+1-m} \, h^{m}k^{n+1-m}\right) .
\end{align*}
\end{thm}

Our third main result is, in turn, a consequence of Theorem~\ref{thm:nReciprocity}: in
conjunction with Theorem~\ref{thm:bettinconrey}, it implies the following explicit formulas for
$\psi_a(z)$ and $g_a(z)$ when $a$ is an odd negative integer.

\begin{thm}\label{thm:nPsiG}
If $n>1$ is an odd integer then
for all $z\in\CC \setminus \RR_{\le 0}$
\[
  \psi_{-n}(z) \ = \
  \frac{(2\pi i)^{n}}{\zeta(n)(n+1)!}\sum_{m=0}^{n+1}{n+1 \choose m}B_{m} \, B_{n+1-m} \, z^{m-1}
\]
and
\[
  g_{-n}(z) \ = \
  \frac{(2\pi i)^{n}}{i(n+1)!}\sum_{m=0}^{n}{n+1 \choose m+1} B_{m+1} \, B_{n-m} \, z^{m} .
\]
\end{thm}

In \cite[Theorem~2]{bettinconreyperiodfunctions}, Bettin and Conrey computed the Taylor series of
$g_a(z)$ and remarked that, if $a$ is a negative integer, $\pi \, g_a^{ (m) } (1)$ is a rational
polynomial in $\pi^2$. Theorem~\ref{thm:nPsiG} generalizes this remark.
We will prove Theorems~\ref{thm:GeneralReciprocity}--\ref{thm:nPsiG} in Section~\ref{sec:mainproofs}.
We note that both Theorem~\ref{thm:nReciprocity} and~\ref{thm:nPsiG} can also be derived directly from Theorem~\ref{thm:bettinconrey}.

Our next goal is to study natural generalizations of $c_a(\frac h k)$.
Taking a leaf from Zagier's generalization of $s(h,k)$ to \emph{higher-dimensional Dedekind sums} \cite{zagier}
and its variation involving cotangent derivatives \cite{beckcot},
let $k_{0},k_{1},\dots,k_{n}$ be positive integers such that $(k_{0},k_{j})=1$ for $j=1,2,\dots,n$, let
$m_{0},m_{1},\dots,m_{n}$ be nonnegative integers, $a \ne -1$ a complex number, and define the \textit{generalized Bettin--Conrey sum}
        \[
        c_{a}\left(\begin{array}{c|ccc}
        k_{0} & k_{1} & \cdots & k_{n}\\
        m_{0} & m_{1} & \cdots & m_{n}
        \end{array}\right) \ = \
k_{0}^{a}\sum_{l=1}^{k_{0}-1}\zeta^{(m_{0})}\left(-a,\frac{l}{k_{0}}\right)\prod_{j=1}^{n}\cot^{(m_{j})}\left(\frac{\pi k_{j}l}{k_{0}}\right) .
        \]
        Here $\zeta^{(m_{0})}(a,z)$ denotes the $m_{0}\th$ derivative of the Hurwitz zeta function with respect to~$z$.
        
This notation mimics that of Dedekind cotangent sums; note that
\[
c_{s}\left(\frac{h}{k}\right) \ = \ c_{s}\left(\begin{array}{c|c}
k & h\\
0 & 0
\end{array}\right).
\]
In Section~\ref{subsec:GeneralizationBettinConreySums}, we will prove reciprocity theorems for generalized
Bettin--Conrey sums, paralleling Theorems~\ref{thm:GeneralReciprocity} and~\ref{thm:nReciprocity}, as well as
more special cases that give, we think, interesting identities.

Our final goal is to relate the particular generalized Bettin--Conrey sum
\[
\sum_{m=1}^{q-1}\cot^{(k)}(\pi mx)\, \zeta \left(-a, \tfrac m q \right)
\]
with evaluations of the \emph{Estermann zeta function}
$
\sum_{n\geq 1}\sigma_a(n) \, \frac{ e^{ 2 \pi i n x } }{ n^s }
$
at integers~$s$; see Section~\ref{sec:estermann}.

        
\section{Proofs of Main Results}\label{sec:mainproofs}

In order to prove Theorem~\ref{thm:GeneralReciprocity}, we need two lemmas.

\begin{lem}
        \label{lem:AsympCot}Let $m$ be a nonnegative integer. Then 
        \[
        \lim_{y\rightarrow\infty}\cot^{(m)}\pi(x\pm iy) \ = \ \begin{cases}
        \mp i & \textnormal{if \ensuremath{m=0},}\\
        0 & \textnormal{if \ensuremath{m>0}.}
        \end{cases}
        \]
        Furthermore, this convergence is uniform with respect to $x$ in a
        fixed bounded interval.\end{lem}
\begin{proof}
        Since $\cot z=\frac{i(e^{iz}+e^{-iz})}{e^{iz}-e^{-iz}}$, we may estimate
        \[
        \left|i+\cot\pi(x+iy)\right| \ = \ \frac{2}{\left|e^{i(2\pi x)}-e^{2\pi y}\right|} \ \leq \
\frac{2}{\left|\left|e^{i(2\pi x)}\right|-\left|e^{2\pi y}\right|\right|} \ = \ \frac{2}{\left|1-e^{2\pi
y}\right|} \, .
        \]
        Given that the rightmost term in this inequality vanishes as $y\rightarrow\infty$,
        we see that \[\lim_{y\rightarrow\infty}\cot\pi(x+iy) \ = \ -i \, .
        \]
        Similarly,
        the inequality 
        \[
        \left|-i+\cot\pi(x-iy)\right| \ = \ \frac{2}{\left|e^{i(2\pi x)}e^{2\pi y}-1\right|} \ \leq \
\frac{2}{\left|\left|e^{i(2\pi x)}e^{2\pi y}\right|-1\right|} \ = \ \frac{2}{\left|e^{2\pi y}-1\right|}
        \]
        implies that $\lim_{y\rightarrow\infty}\cot\pi(x-iy)=i$. Since
        \[
        \left|\csc\pi(x+iy)\right| \ = \ \frac{2e^{\pi y}}{\left|e^{i\pi x}-e^{-i\pi x}e^{2\pi y}\right|} \ \leq
\ \frac{2e^{\pi y}}{\left|\left|e^{i\pi x}\right|-\left|e^{-i\pi x}e^{2\pi y}\right|\right|} \ = \ \frac{2e^{\pi
y}}{\left|1-e^{2\pi y}\right|} \, ,
        \]
        it follows that $\lim_{y\rightarrow\infty}\csc\pi(x+iy)=0$. Similarly, 
        \[
        \left|\csc\pi(x-iy)\right| \ = \ \frac{2e^{\pi y}}{\left|e^{i\pi x}e^{2\pi y}-e^{-i\pi x}\right|} \ \leq
\ \frac{2e^{\pi y}}{\left|\left|e^{i\pi x}e^{2\pi y}\right|-\left|e^{-i\pi x}\right|\right|} \ = \ \frac{2e^{\pi y}}{\left|e^{2\pi y}-1\right|}
        \]
        implies that $\lim_{y\rightarrow\infty}\csc\pi(x-iy)=0$. We remark that $\frac{d}{dz}(\cot z)=-\csc^{2}z$ and 
        \[\frac{d}{dz}(\csc z) \ = \ -\csc z\cot z \, ,
        \]
        so all the derivatives of $\cot z$ have a $\csc z$ factor, and therefore, 
        \[
        \lim_{y\rightarrow\infty}\cot^{(m)}\pi(x\pm iy) \ = \ \begin{cases}
        \mp i & \textrm{if \ensuremath{m=0,}}\\
        0 & \textrm{if \ensuremath{m>0}.}
        \end{cases}
        \]
        Since the convergence above is independent of $x$, the limit is uniform with respect to $x$ in a fixed bounded interval. 
\end{proof}
Lemma \ref{lem:AsympCot} implies that 
\[
\lim_{y\rightarrow\infty}\cot^{(m)}\pi h(x\pm iy)
\ = \ \lim_{y\rightarrow\infty}\cot^{(m)}\pi k(x\pm iy)
\ = \ \begin{cases}
  \mp i & \textrm{if \ensuremath{m=0,}}\\
      0 & \textrm{if \ensuremath{m>0},}
\end{cases}
\]
uniformly with respect to $x$ in a fixed bounded interval.

The proof of the following lemma is hinted at by Apostol~\cite{MR0046379}.
\begin{lem}
        \label{lem:AsympHurwitz} If $\Re(a)>1$ and $R>0$, then $\zeta(a,x+iy)$
        vanishes uniformly with respect to $x\in[0,R]$ as $y\rightarrow\pm\infty$.\end{lem}

\begin{proof}
        We begin by showing that $\zeta(a,z)$ vanishes as $\Im(z)\rightarrow\pm\infty$
        if $\Re(z)>0$. Since $\Re(a)>1$ and $\Re(z)>0$, we have the integral
        representation \cite[eq.~25.11.25]{MR2723248}
        \[
        \zeta(a,z) \ = \ \frac{1}{\Gamma(a)}\int_{0}^{\infty}\frac{t^{a-1}e^{-z\, t}}{1-e^{-t}}dt,
        \]
        which may be written as
        \begin{equation}
        \zeta(a,z) \ = \ \frac{1}{\Gamma(a)}\int_{0}^{\infty}\frac{t^{a-1}e^{-t\,\Re(z)}}{1-e^{-t}}e^{-it\Im(z)}dt.\label{eq:hurwitzintegral-1}
        \end{equation}
        Note that for fixed $\Re(z)$,
        \[
        \int_{0}^{\infty}\frac{t^{a-1}e^{-t\Re(z)}}{1-e^{-t}}dt \ = \ \zeta \left(a,\Re(z)\right) \Gamma(a)
        \]
        and 
        \[
        \int_{0}^{\infty}\left|\frac{t^{a-1}e^{-t\,\Re(z)}}{1-e^{-t}}\right|dt \ = \
\int_{0}^{\infty}\frac{t^{\Re(a)-1}e^{-t\,\Re(z)}}{1-e^{-t}}dt \ = \ \zeta \left( \Re(a),\Re(z) \right)
\Gamma(\Re(a)) \, ,
        \]
        so the Riemann--Lebesgue lemma (see, for example, \cite[Theorem  ~16]{lighthill1958introduction}) implies that 
        \[\int_{0}^{\infty}\frac{t^{a-1}e^{-t\Re(z)}}{1-e^{-t}}e^{-it\Im(z)}dt\]
        vanishes as $\Im(z)\rightarrow\pm\infty$. By (\ref{eq:hurwitzintegral-1}),
        this means that for $\Re(z)$ fixed, $\zeta(a,z)$ vanishes as $\Im(z)\rightarrow\pm\infty$.
        In other words, $\zeta(a,x+iy)\rightarrow0$ pointwise with respect
        to $x>0$ as $y\rightarrow\pm\infty$. 
        
        Moreover, the vanishing of $\zeta(a,x+iy)$ as $y\rightarrow\pm\infty$
        is uniform with respect to $x\in[0,R]$. Indeed, denote $g(t)=\frac{t^{a-1}e^{-tR}}{1-e^{-t}}$,
        then (\ref{eq:hurwitzintegral-1}) implies that
        \[
        \int_{0}^{\infty}g(t)dt \ = \ \Gamma(a)\,\zeta(a,R)
        \]
        and 
        \[
        \int_{0}^{\infty}\left|g(t)\right|dt \ = \ \Gamma(\Re(a))\,\zeta(\Re(a),R) \, .
        \]
        It then follows from the Riemann--Lebesgue lemma that $\lim_{\left|z\right|\rightarrow\infty}\int_{0}^{\infty}g(t)e^{-itz}dt=0$.
        If $x\in(0,R]$, we may write
        \[
        \Gamma(a)\,\zeta(a,x\pm iy) \ = \ \int_{0}^{\infty}\frac{t^{a-1}e^{-tx}}{1-e^{-t}}\,e^{\mp ity}dt \ = \ \int_{0}^{\infty}g(t)\,e^{-it(\pm y-i(x-R))}dt.
        \]
        
        Since $g(t)$ does not depend on $x$, the speed at which $\zeta(a,x\pm iy)$
        vanishes depends on $R$ and $y^{2}+(x-R)^{2}$. However,
        we know that $0\leq\left|x-R\right|<R$, so the speed of the vanishing
        depends only on $R$. 
        
        Finally, note that 
        \[
        \zeta(a,iy) \ = \ \sum_{n=0}^{\infty}\frac{1}{(iy+n)^{a}} \ = \
\sum_{n=0}^{\infty}\frac{1}{(1+iy+n)^{a}}+\frac{1}{(iy)^{a}} \ = \ \zeta(a,1+iy)+\frac{1}{(iy)^{a}} \, ,
        \]
        so $\zeta(a,iy)\rightarrow0$ as $y\rightarrow\pm\infty$,
        and the speed at which $\zeta(s,iy)$ vanishes depends on that of
        $\zeta(s,1+iy)$. Thus, $\zeta(s,x+iy)\rightarrow0$ uniformly
        as $y\rightarrow\pm\infty$, as long as $x\in[0,R]$.
\end{proof}

\begin{proof}[Proof of Theorem \ref{thm:GeneralReciprocity}]
The idea is to use Cauchy's residue theorem to integrate the function 
\[
f(z) \ = \ \cot(\pi hz)\cot(\pi kz)\,\zeta(a,z)
\]
along $C(M,\epsilon)$ as $M\rightarrow\infty$, where $C(M,\epsilon)$
denotes the positively oriented rectangle with vertices $1+\epsilon+iM$,
$\epsilon+iM$, $\epsilon-iM$ and $1+\epsilon-iM$, for $M>0$ and
$0<\epsilon<\min\left\{ \frac{1}{h},\frac{1}{k}\right\} $ (see Figure~\ref{fig:IntegrationPath}).

\begin{figure}[htb]
        \noindent \begin{centering}
                \includegraphics[scale=0.5]{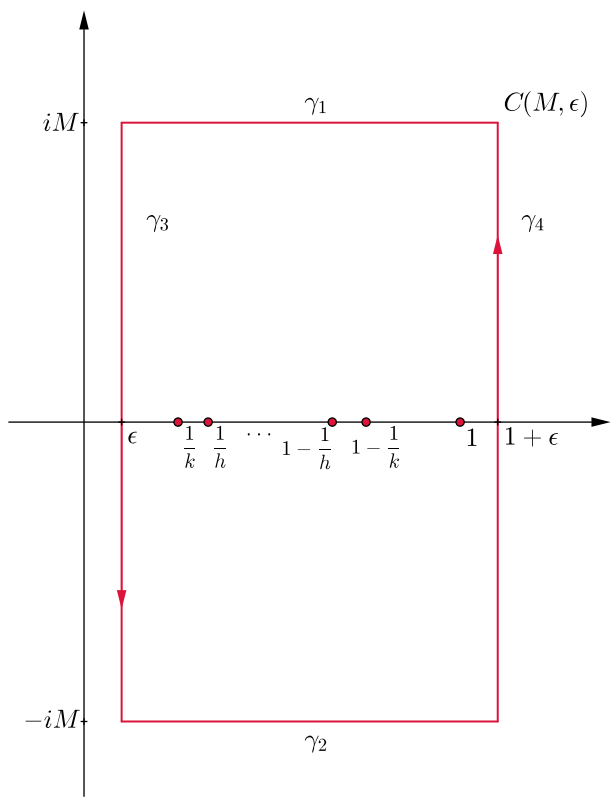}
                \par\end{centering}
        
        \protect\caption{The closed contour $C(M,\epsilon).$\label{fig:IntegrationPath}}
\end{figure}

Henceforth, $a\in\mathbb{C}$ is such that $\Re(a)>1$,
$(h,k)$ is a pair of coprime positive intergers, and $f(z)$ and
$C(M,\epsilon)$ are as above, unless otherwise stated.
        Since $\zeta(a,z)$
        is analytic inside $C(M,\epsilon)$, the only poles of $f(z)$ are
        those of the cotangent factors. Thus, the fact that $h$ and $k$
        are coprime implies that a complete list of the possible poles of
        $f(z)$ inside $C(M,\epsilon)$ is 
        \[
        E=\left\{ \frac{1}{h},\dots,\frac{h-1}{h},\frac{1}{k},\dots,\frac{k-1}{k},1\right\} 
        \]
        and each of these poles is (at most) simple, with the exception of
        1, which is (at most) double. 
        For $m\in\{1,2,\dots,h-1\}$,
        \begin{align*}
        \res_{z=\frac{m}{h}}f(z)
        \ &= \ \cot\left(\frac{\pi km}{h}\right)\cos(\pi m)\,\zeta\left(a,\frac{m}{h}\right)\res_{z=\frac{m}{h}}\frac{1}{\sin(\pi hz)} \\
        \ &= \ \frac{1}{\pi h}\cot\left(\frac{\pi km}{h}\right)\zeta\left(a,\frac{m}{h}\right).
        \end{align*} 
        Of course, an analogous result is true for $\res_{z=\frac{m}{k}}f(z)$
        for all $m\in\{1,2,\dots,k-1\}$, and therefore
        \begin{align*}
        &\sum_{z_{0}\in E}\res_{z=z_{0}}f(z) \ = \\
        &\qquad \res_{z=1}f(z)+\frac{1}{\pi h}\sum_{m=1}^{h-1}\cot\left(\frac{\pi km}{h}\right)\zeta\left(a,\frac{m}{h}\right)
        +\frac{1}{\pi k}\sum_{m=1}^{k-1}\cot\left(\frac{\pi hm}{k}\right)\zeta\left(a,\frac{m}{k}\right)
        \end{align*}
        or, equivalently,
        \begin{equation}
        h^{1-a}c_{-a}\left(\frac{h}{k}\right)+k^{1-a}c_{-a}\left(\frac{k}{h}\right) \ = \ \pi(hk)^{1-a}\left(\left(\sum_{z_{0}\in
E}\res_{z=z_{0}}f(z)\right)-\res_{z=1}f(z)\right).\label{eq:RecipResidue}
        \end{equation}
        
        We now determine $\res_{z=1}f(z)$. The Laurent series of the cotangent function
        about 0 is given by
        \[
        \cot z \ = \ \frac{1}{z}-\frac{1}{3}z-\frac{1}{45}z^{3}-\frac{2}{945}z^{5}+\cdots,
        \]
        so, by the periodicity of $\cot z$, for $z\neq1$ in a small neighborhood of $z=1$,
        \[
        \cot(\pi kz) \ = \ \left(\frac{1}{\pi k}\right)\frac{1}{z-1}-\frac{\pi k}{3}(z-1)-\frac{(\pi k)^{3}}{45}(z-1)^{3}-\frac{2(\pi k)^{5}}{945}(z-1)^{5}+\cdots
        \]
        and, similarly,
        \[
        \cot(\pi hz) \ = \ \left(\frac{1}{\pi h}\right)\frac{1}{z-1}-\frac{\pi h}{3}(z-1)-\frac{(\pi h)^{3}}{45}(z-1)^{3}-\frac{2(\pi h)^{5}}{945}(z-1)^{5}+\cdots.
        \]
        
        Since $\zeta(a,z)$ is analytic in a small neighborhood of $1$, Taylor's
        theorem implies that 
        \[\zeta(a,z)=\sum_{n=0}^{\infty}b_{n}(z-1)^{n},
        \]
        where $b_{n}=\frac{\zeta^{(n)}(a,1)}{n!}$ for $n=0,1,2,\dots$ (derivatives
        relative to $z$). Thus, the expansion of $f(z)$ about 1 is of the
        form
        \[
        \frac{b_{0}}{\pi^{2}hk}\left(\frac{1}{z-1}\right)^{2}+\left(\frac{b_{1}}{\pi^{2}hk}\right)\frac{1}{z-1}+(\textrm{analytic part}).
        \]
        
        Given that $a\neq0,1$, we know that $\frac{\partial}{\partial z}\zeta(a,z)=-a\,\zeta(a+1,z)$
        \cite[eq.~25.11.17]{MR2723248}, so $b_{1}=-a\,\zeta(a+1,1)=-a\,\zeta(a+1)$.
        We conclude that $\res_{z=1}f(z)=-\frac{a\,\zeta(a+1)}{\pi^{2}hk}$
        and it then follows from (\ref{eq:RecipResidue}) that 
        \begin{equation}
        h^{1-a}c_{-a}\left(\frac{h}{k}\right)+k^{1-a}c_{-a}\left(\frac{k}{h}\right) \ = \ \frac{a\,\zeta(a+1)}{\pi(hk)^{a}}+\frac{\pi}{(hk)^{a-1}}\sum_{z_{0}\in
E}\res_{z=z_{0}}f(z).\label{eq:ResidueSide}
        \end{equation}
        
        We now turn to the computation of $\displaystyle\sum_{z_{0}\in E}\res_{z=z_{0}}f(z)$
        via Cauchy's residue theorem, which together with (\ref{eq:ResidueSide})
        will provide the reciprocity we are after. Note that the function
        $f(z)$ is analytic on any two closed contours $C(M_{1},\epsilon)$
        and $C(M_{2},\epsilon)$ and since the poles inside these two contours
        are the same, we may apply Cauchy's residue theorem to both contours
        and deduce that
        \[
        \int_{C(M_{1},\epsilon)}f(z) \, dz \ = \ \int_{C(M_{2},\epsilon)}f(z) \, dz \, .
        \]
        In particular, this implies that 
        \begin{equation}
        \lim_{M\rightarrow\infty}\int_{C(M,\epsilon)}f(z) \, dz \ = \ 2\pi i\sum_{z_{0}\in E}\res_{z=z_{0}}f(z) \, . \label{eq:CauchyLim}
        \end{equation}
        
        Let $\gamma_{1}$ be the path along $C(M,\epsilon)$ from
        $1+\epsilon+iM$ to $\epsilon+iM$. Similarly, define $\gamma_{2}$
        from $\epsilon-iM$ to $1+\epsilon-iM$, $\gamma_{3}$
        from $\epsilon+iM$ to $\epsilon-iM$, and $\gamma_{4}$
        from $1+\epsilon-iM$ to $1+\epsilon+iM$ (see Figure \ref{fig:IntegrationPath}).
        Since $\Re(a)>1$,
        \[
        \zeta(a,z+1) \ = \ \sum_{n=0}^{\infty}\frac{1}{(n+z+1)^{a}} \ = \ \sum_{n=1}^{\infty}\frac{1}{(n+z)^{a}} \ = \
\zeta(a,z)-\frac{1}{z^{a}} \, ,
        \]
        and so the periodicity of $\cot z$ implies that
         \[
         \int_{\gamma_{4}}f(z) \, dz \ = \ -\int_{\gamma_{3}}f(z) \, dz+\int_{\gamma_{3}}\frac{\cot(\pi hz)\cot(\pi kz)}{z^{a}} \, dz \, .
         \]
        Lemmas \ref{lem:AsympCot} and \ref{lem:AsympHurwitz} imply that $f(z)$ vanishes uniformly as $M\rightarrow\infty$ (uniformity
        with respect to $\Re(z)\in[\epsilon,1+\epsilon]$) so
        \[
        \lim_{M\rightarrow\infty}\int_{\gamma_{1}}f(z) \, dz \ = \ 0 \ = \ \lim_{M\rightarrow\infty}\int_{\gamma_{2}}f(z) \, dz \, .
        \]
        This means that
        \[
         \lim_{M\rightarrow\infty}\int_{C(M,\epsilon)}f(z) \, dz \ = \ \lim_{M\rightarrow\infty}\left(\int_{\gamma_{3}}f(z) \, dz+\int_{\gamma_{4}}f(z) \, dz\right)
         \]
        and it follows from (\ref{eq:ResidueSide}) and (\ref{eq:CauchyLim}) that
        \begin{align*}
        &h^{1-a}c_{-a}\left(\frac{h}{k}\right)+k^{1-a}c_{-a}\left(\frac{k}{h}\right) \ = \\
        &\qquad \frac{a\,\zeta(a+1)}{\pi(hk)^{a}}+\frac{(hk)^{1-a}}{2i}\int_{\epsilon+i\infty}^{\epsilon-i\infty}\frac{\cot(\pi hz)\cot(\pi kz)}{z^{a}} \, dz\, .
        \end{align*} 
        This completes the proof of Theorem~\ref{thm:GeneralReciprocity}. 
\end{proof}

To prove Theorem \ref{thm:nReciprocity}, we now turn to the particular case in which $a=n>1$ is an odd integer and study Bettin--Conrey sums of the form $c_{-n}$. 
        
Let $\Psi^{(n)}(z)$ denote the $(n+2)$-th polygamma function (see, for example, \cite[Sec.~5.15]{MR2723248}). It is well known that for $n$ a positive integer,
 \[
 \zeta(n+1,z) \ = \ \frac{(-1)^{n+1}\Psi^{(n)}(z)}{n!}
 \]
whenever $\Re(z)>0$ (see, for instance, \cite[eq.~25.11.12]{MR2723248}),
so for $n>1$, we may write 
\[
c_{-n}\left(\frac{h}{k}\right) \ = \ \frac{(-1)^{n}}{k^{n}(n-1)!}\sum_{m=1}^{k-1}\cot\left(\frac{\pi mh}{k}\right)\Psi^{(n-1)}\left(\frac{m}{k}\right).
\]
By the reflection formula for the polygamma functions \cite[eq.~5.15.6]{MR2723248},
\[
\Psi^{(n)}(1-z)+(-1)^{n+1}\Psi^{(n)}(z) \ = \ (-1)^{n}\pi\cot^{(n)}(\pi z),
\]
we know that if $n$ is odd, then 
\[
 \Psi^{(n-1)}\left(1-\frac{m}{k}\right)-\Psi^{(n-1)}\left(\frac{m}{k}\right) \ = \ \pi\cot^{(n-1)}\left(\frac{\pi m}{k}\right)
\]
for each $m\in\{1,2,\dots,k-1\}$. Therefore, 
\begin{align*}
                2\sum_{m=1}^{k-1}\cot\left(\frac{\pi mh}{k}\right)\Psi^{(n-1)}\left(\frac{m}{k}\right) \ &= \ \sum_{m=1}^{k-1}\cot\left(\frac{\pi
mh}{k}\right)\Psi^{(n-1)}\left(\frac{m}{k}\right)\\
                &\qquad +\sum_{m=1}^{k-1}\cot\left(\frac{\pi(k-m)h}{k}\right)\Psi^{(n-1)}\left(1-\frac{m}{k}\right),
\end{align*}
which implies that
\begin{align*}
        &2\sum_{m=1}^{k-1}\cot\left(\frac{\pi mh}{k}\right)\Psi^{(n-1)}\left(\frac{m}{k}\right) \\
        &\qquad = \ \sum_{m=1}^{k-1}\cot\left(\frac{\pi mh}{k}\right)\left(\Psi^{(n-1)}\left(\frac{m}{k}\right)-\Psi^{(n-1)}\left(1-\frac{m}{k}\right)\right)\\
        &\qquad = \ -\pi\sum_{m=1}^{k-1}\cot\left(\frac{\pi mh}{k}\right)\cot^{(n-1)}\left(\frac{\pi m}{k}\right).
\end{align*}
This means that for $n>1$ odd, $c_{-n}$ is essentially a Dedekind cotangent sum. Indeed, using the notation in~\cite{beckcot},
\begin{align*}
c_{-n}\left(\frac{h}{k}\right) \ &= \ \frac{\pi}{2k^{n}(n-1)!}\sum_{m=1}^{k-1}\cot\left(\frac{\pi mh}{k}\right)\cot^{(n-1)}\left(\frac{\pi m}{k}\right)\\
&= \ \frac{\pi}{2(n-1)!}\,\mathfrak{c}\left(\begin{array}{c|cc}
k & h & 1\\
n-1 & 0 & n-1\\
0 & 0 & 0\\
\end{array}\right).
\end{align*}

Thus Theorem~\ref{thm:nReciprocity} is an instance of Theorem~\ref{thm:GeneralReciprocity}.  Its significance is a reciprocity instance for Bettin--Conrey sums of the form
$c_{-n}$ in terms of Bernoulli numbers. For this reason we give the details of its proof. 

\begin{proof}[Proof of Theorem \ref{thm:nReciprocity}]
        We consider the closed contour $\widetilde{C}(M,\epsilon)$ defined
        as the positively oriented rectangle with vertices $1+iM$, $iM$,
        $-iM$ and $1-iM$, with indentations (to the right) of radius $0<\epsilon<\min\left\{ \frac{1}{h},\frac{1}{k}\right\} $
        around 0 and 1 (see Figure~\ref{fig:IndentedPath}). 
        \begin{figure}[htb]
                \noindent \begin{centering}
                        \includegraphics[scale=0.5]{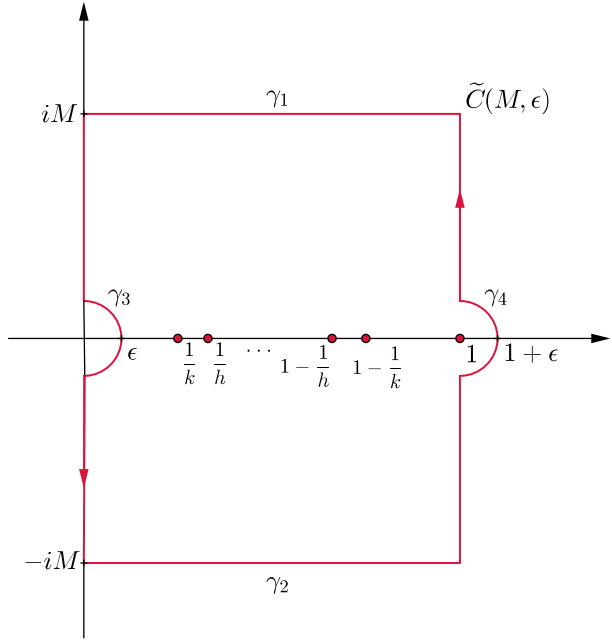}
                        \par\end{centering}
                
                \protect\caption{The closed contour $\widetilde{C}(M,\epsilon).$\label{fig:IndentedPath}}
        \end{figure}
        
        Since $\widetilde{C}(M,\epsilon)$ contains the same poles of $f(z) = \cot(\pi hz)\cot(\pi kz)\,\zeta(n,z)$
        as the closed contour $C(M,\epsilon)$ in Figure \ref{fig:IntegrationPath} used to prove Theorem \ref{thm:GeneralReciprocity}, we
may apply Cauchy's residue theorem, letting $M\rightarrow\infty$, and we only need to determine $\lim_{M\rightarrow\infty}\int_{\widetilde{C}(M,\epsilon)}f(z) \, dz$ in order
to deduce a reciprocity law for the sums $c_{-n}$.
        
        As in the case of $C(M,\epsilon)$, the integrals along the horizontal
        paths vanish, so using the periodicity of the cotangent to add integrals
        along parallel paths, as we did when considering $C(M,\epsilon)$,
        we obtain 
        \begin{equation}
        \lim_{M\rightarrow\infty}\int_{\widetilde{C}(M,\epsilon)}f(z) \, dz \ = \
\lim_{M\rightarrow\infty}\left(\int_{iM}^{i\epsilon}g(z) \, dz+\int_{-i\epsilon}^{-iM}g(z) \, dz\right)+\int_{\gamma_3}g(z) \, dz \, ,\label{eq:verticIntegralCancel}
        \end{equation}
        where $\gamma_3$ denotes the indented path around 0 and
        \[
        g(z) \ = \ \frac{\cot(\pi kz)\cot(\pi hz)}{z^{n}} \, .
        \]
        Given that $g(z)$ is an odd function,
        the vertical integrals cancel and we may apply Cauchy's residue theorem
        to integrate $g(z)$ along the positively oriented circle of radius $\epsilon$
        and centered at 0, to deduce that
        \[
        \lim_{M\rightarrow\infty}\int_{\widetilde{C}(M,\epsilon)}f(z) \, dz \ = \ -\pi i\res_{z=0}g(z) \, .
        \]
        This is the main reason to use the contour $\widetilde{C}(M,\epsilon)$ instead of $C(M,\epsilon)$. Indeed, integration along $\widetilde{C}(M,\epsilon)$ exploits the
parity of the function $g(z)$, allowing us to cancel the vertical integrals in (\ref{eq:verticIntegralCancel}).
        
        The expansion of the cotangent function is
        \[\pi z\cot(\pi z) \ = \ \sum_{m=0}^{\infty}\frac{(2\pi i)^{m}B_{m}}{m!}z^{m},\]
        with the convention that $B_{1}$ must be redefined to be zero. Thus,
        we have the expansion 
        \[
        \cot(\pi kz) \ = \ \sum_{m=-1}^{\infty}\frac{(2i)(2\pi ik)^{m}B_{m+1}}{(m+1)!}z^{m}
        \]
        and of course, an analogous result holds for $h$. Hence, 
        \begin{equation}
        \res_{z=0}g(z) \ = \ \frac{(2i)(2\pi i)^{n}}{\pi hk(n+1)!}\sum_{m=0}^{n+1}{n+1 \choose m}B_{m}B_{n+1-m}h^{m}k^{n+1-m},\label{eq:Resg(z)}
        \end{equation}
        and given that $\zeta(n+1)=-\frac{(2\pi i)^{n+1}}{2(n+1)!}B_{n+1}$
        \cite[eq.~25.6.2]{MR2723248}, the Cauchy residue theorem and
        (\ref{eq:ResidueSide}) yield
        \begin{align}
        &\left(\frac{2\pi i}{hk}\right)^{n}\frac{1}{i(n+1)!}\left(nB_{n+1}+\sum_{m=0}^{n+1}{n+1 \choose m}B_{m}B_{n+1-m}h^{m}k^{n+1-m}\right) \label{eq:RecipProofOdd} \\
        &\qquad = \ h^{1-n}c_{-n}\left(\frac{h}{k}\right)+k^{1-n}c_{-n}\left(\frac{k}{h}\right). \nonumber
        \end{align}
        
        Finally, note that the convention $B_{1}:=0$ is irrelevant in (\ref{eq:RecipProofOdd}),
        since $B_{1}$ in this sum is always multiplied by a Bernoulli number with odd index larger than~1.
\end{proof}

Note that Theorem \ref{thm:nReciprocity} is essentially the same as 
 the reciprocity deduced by Apostol for Dedekind--Apostol sums \cite{MR0034781}. This is a consequence
of the fact that for $n>1$ an odd integer, $c_{-n}\left(\frac{h}{k}\right)$
is a multiple of the Dedekind--Apostol sum $s_{n}(h,k)$. Indeed,
for such $n$ \cite[Theorem  ~1]{MR0046379}
\[
s_{n}(h,k) \ = \ i \, n! \, (2\pi i)^{-n} \, c_{-n}\left(\frac{h}{k}\right).
\]
It is worth mentioning that although the Dedekind--Apostol sum
$s_{n}(h,k)$ is trivial for $n$ even \cite[eq.~(4.13)]{MR0034781}, in the sense that $s_{n}(h,k)$ is independent of $h$, the
Bettin--Conrey sum $c_{-n}\left(\frac{h}{k}\right)$ is not.

The following corollary is an immediate consequence of Theorems~\ref{thm:GeneralReciprocity} and~\ref{thm:nReciprocity}.
\begin{cor}
        \label{cor:nIntegral}Let $n>1$ be an odd integer and suppose $h$
        and $k$ are positive coprime integers, then for any $0<\epsilon<\min\left\{ \frac{1}{h},\frac{1}{k}\right\}$,
        \[
        \int_{\epsilon+i\infty}^{\epsilon-i\infty}\frac{\cot(\pi hz)\cot(\pi kz)}{z^{n}} \, dz \ = \ \frac{2(2\pi i)^{n}}{hk(n+1)!}\sum_{m=0}^{n+1}{n+1 \choose
m}B_{m}B_{n+1-m}h^{m}k^{n+1-m}.
        \]
\end{cor}



\begin{proof}[Proof of Theorem \ref{thm:nPsiG}]
        Given that $c_{-n}\left(\frac{-k}{h}\right)=-c_{-n}\left(\frac{k}{h}\right)$,
        it follows from Theorems~\ref{thm:bettinconrey} and~\ref{thm:nReciprocity}
        that 
        \begin{equation}
        \psi_{-n}\left(\frac{h}{k}\right) \ = \ \frac{(2\pi i)^{n}}{\zeta(n)(n+1)!}\sum_{m=0}^{n+1}{n+1 \choose m}B_{m}B_{n+1-m}\left(\frac{h}{k}\right)^{m-1}.\label{eq:Psihk}
        \end{equation}
        The function 
        \[
        \phi_{-n}(z) \ = \ \frac{(2\pi i)^{n}}{\zeta(n)(n+1)!}\sum_{m=0}^{n+1}{n+1 \choose m}B_{m}B_{n+1-m} \, z^{m-1}
        \]
        is analytic on $\mathbb{C}\backslash\mathbb{R}_{\leq0}$
        and, by \cite[Theorem~1]{bettinconreyperiodfunctions}, so is $\psi_{-n}$. Let 
        \[
        S_{n} \ = \ \left\{ z\in\mathbb{C}\backslash\mathbb{R}_{\leq0}\mid\psi_{-n}(z)=\phi_{-n}(z)\right\}.
        \]
        Since all positive rationals can be written in reduced form, it follows from (\ref{eq:Psihk}) that $\mathbb{Q}_{>0}\subseteq S_{n}$. Thus, $S_{n}$ is
        not a discrete set and given that both $\psi_{-n}$ and $\phi_{-n}$
        are analytic on the connected open set $\mathbb{C}\backslash\mathbb{R}_{\leq0}$,
        Theorem 1.2($ii$) in \cite[p.~90]{MR1659317} implies that $\psi_{-n}=\phi_{-n}$
        on $\mathbb{C}\backslash\mathbb{R}_{\leq0}$. That is, 
        \[
        \psi_{-n}(z) \ = \ \frac{(2\pi i)^{n}}{\zeta(n)(n+1)!}\sum_{m=0}^{n+1}{n+1 \choose m}B_{m}B_{n+1-m} \, z^{m-1}
        \]
        for all $z\in\mathbb{C}\backslash\mathbb{R}_{\leq0}$. Now
        \[
        \psi_{-n}(z) \ = \ \frac{i}{\pi z}\frac{\zeta(1+n)}{\zeta(n)}-iz^{n-1}\cot\left(\frac{-\pi n}{2}\right)+i\frac{g_{-n}(z)}{\zeta(n)},
        \]
        and since $n$ is odd, $ $$\cot\left(\frac{-\pi n}{2}\right)=0$
        and $\zeta(n+1)=-\frac{(2\pi i)^{n+1}}{2(n+1)!}B_{n+1}$ \cite[eq.~25.6.2]{MR2723248}, so
        
        \begin{align*}
                g_{-n}(z) \ &= \ \frac{i(2\pi i)^{n}B_{n+1}}{(n+1)!}\left(\frac{1}{z}\right)-i\,\zeta(n)\,\psi_{-n}(z)\\
                &= \ \frac{-i(2\pi i)^{n}}{(n+1)!}\sum_{m=1}^{n+1}{n+1 \choose m}B_{m}B_{n+1-m} \, z^{m-1}\\
                &= \ \frac{-i(2\pi i)^{n}}{(n+1)!}\sum_{m=0}^{n}{n+1 \choose m+1}B_{m+1}B_{n-m} \, z^{m}\, . \qedhere
        \end{align*}
\end{proof}

Clearly, Theorem \ref{thm:nPsiG} is a particular case of Bettin--Conrey's \cite[Theorem  ~3]{bettinconreyperiodfunctions}. However, the
proofs  are independent, so Theorem~\ref{thm:nPsiG} is stronger (in the particular case $a=-n$, with $n>1$ an odd integer), because it
completely determines $g_{-n}$ and shows that $g_{-n}$ is a polynomial. 
In particular, it becomes obvious that if $a\in\mathbb{Z}_{\leq1}$ is odd and $(a,m)\neq(0,0)$,
then $\pi g_{a}^{(m)}(1)$ is a rational polynomial in~$\pi^{2}$. 


\section{Generalizations of Bettin--Conrey Sums}\label{subsec:GeneralizationBettinConreySums}

Now we study generalized Bettin--Conrey sum
        \[
        c_{a}\left(\begin{array}{c|ccc}
        k_{0} & k_{1} & \cdots & k_{n}\\
        m_{0} & m_{1} & \cdots & m_{n}
        \end{array}\right) \ = \
k_{0}^{a}\sum_{l=1}^{k_{0}-1}\zeta^{(m_{0})}\left(-a,\frac{l}{k_{0}}\right)\prod_{j=1}^{n}\cot^{(m_{j})}\left(\frac{\pi k_{j}l}{k_{0}}\right) .
        \]      
Henceforth, $B_n$ denotes the $n$-th Bernoulli number with the convention $B_{1}:=0$. 

The following reciprocity theorem generalizes Theorem~\ref{thm:GeneralReciprocity}.
\begin{thm}
        \label{thm:GeneralRecip}Let $d\geq2$ and suppose that $k_{1},\dots,k_{d}$
        is a list of pairwise coprime positive integers and $m_{0},m_{1},\dots,m_{d}$
        are nonnegative integers. If $\Re(a)>1$ and $0<\epsilon<\min_{1\leq j\leq d}\left\{ \frac{1}{k_{j}}\right\} $
        then 
\begin{multline*}
        \sum_{j=1}^{d}\frac{(-1)^{m_{j}}}{\pi}\sum_{{l_{0}+\cdots+\widehat{l_{j}}+\cdots+l_{d}=m_{j}\atop l_{0},\ldots,\widehat{l_{j}},\ldots,l_{d}\geq0}}{m_{j} \choose
l_{0},\ldots,\widehat{l_{j}},\ldots,l_{d}}\left(\prod_{{t=1\atop t\neq j}}^{d}(\pi k_{t})^{l_{t}}\right)k_{j}^{a-1}\,c_{-a}(j) \ = \\
        -\left(\sum_{l_{0}=0}^{m_{1}+\cdots+m_{d}+d-1}\sum_{l_{1}+\cdots+l_{d}=-l_{0}-1}\prod_{j=0}^{d}a_{l_{j}}\right)+\frac{(-1)^{m_{0}}a^{(m_{0})}}{2\pi
i}\int_{\epsilon+i\infty}^{\epsilon-i\infty}\frac{\prod_{j=1}^{d}\cot^{(m_{j})}(\pi k_{j}z)}{z^{a+m_{0}}}dz \, .
\end{multline*}
        where $x^{(n)}=\prod_{l=0}^{n-1}(x+l)$ is the rising factorial,
        \[
        c_{-a}(j) \ = \ c_{-a}\left(\begin{array}{c|ccccc}
        k_{j} & k_{1} & \cdots & \widehat{k_{j}} & \cdots & k_{d}\\
        m_{0}+l_{0} & m_{1}+l_{1} & \cdots & \widehat{m_{j}+l_{j}} & \cdots & m_{d}+l_{d}
        \end{array}\right) ,
        \]
        and for $j=0,1,\ldots,d$, we define 
\[
a_{l_{j}}=\begin{cases}
\frac{(-1)^{m_{0}+l_{0}}a^{(m_{0}+l_{0})}\zeta(a+m_{0}+l_{0})}{l_{0}!} & \textnormal{if }j=0\textnormal{ and }l_{0}\geq0,\\
\frac{(2i)^{l_{j}+m_{j}+1}B_{l_{j}+m_{j}+1}(\pi k_{j})^{l_{j}+m_{j}}(l_{j}+1)^{(m_{j})}}{(l_{j}+m_{j}+1)!} & \textnormal{if }j\neq0\textnormal{ and }l_{j}\geq0,\\
\frac{(-1)^{m_{j}}m_{j}!}{\pi k_{j}} & \textnormal{if }j\neq0\textnormal{ and }l_{j}=-(m_{j}+1),\\
0 & \textnormal{otherwise.}
\end{cases}
\]
\end{thm}

The proof of this theorem is analogous to that of Theorem \ref{thm:GeneralReciprocity}.
Henceforth, $\Re(a)>1$, $k_{1},\dots,k_{d}$ is
a list of pairwise coprime positive integers and $0<\epsilon<\min_{1\leq j\leq d}\left\{ \frac{1}{k_{j}}\right\} $.
In addition, $m_{0},m_{1},\dots,m_{d}$ is a list of nonnegative integers,
\[
f(z) \ = \ \zeta^{(m_{0})}(a,z)\prod_{j=1}^{d}\cot^{(m_{j})}(\pi k_{j}z)
\]
and, as before, $C(M,\epsilon)$ denotes the positively oriented rectangle
with vertices $1+\epsilon+iM$, $\epsilon+iM$, $\epsilon-iM$ and
$1+\epsilon-iM$, where $M>0$ (see Figure \ref{fig:IntegrationPath}).


\begin{proof}
        [Proof of Theorem \ref{thm:GeneralRecip}]For each $j$, we know that
        $\cot\left(\pi k_{j}z\right)$ is analytic on and inside $C(M,\epsilon)$,
        with the exception of the poles $\frac{1}{k_{j}},\dots,\frac{k_{j}-1}{k_{j}}$.
        This means that except for the aforementioned poles, $\cot^{(m_{j})}(\pi k_{j}z)$
        is analytic on and inside $C(M,\epsilon)$, so the analyticity of
        $\zeta^{(m_{0})}(a,z)$ on and inside $C(M,\epsilon)$ implies that
        a complete list of (possible) poles of $f$ is 
        \[
        E \ = \ \left\{ \frac{1}{k_{1}},\dots,\frac{k_{1}-1}{k_{1}},\dots\frac{1}{k_{d}},\dots,\frac{k_{d}-1}{k_{d}},1\right\} .
        \]
        
        Let $j\in\{1,2,\dots d\}$ and $q\in\{1,2,\dots,k_{j}-1\}$, then
        the Laurent series of $\cot(\pi k_{j}z)$ about $\frac{q}{k_{j}}$
        is of the form $\left(\frac{1}{\pi k_{j}}\right)\frac{1}{z-\frac{q}{k_{j}}}+(\textrm{analytic part})$,
        so near $\frac{q}{k_{j}}$,
        \[
        \cot^{(m_{j})}(\pi k_{j}z) \ = \ \frac{(-1)^{m_{j}}m_{j}!}{\pi k_{j}}\left(z-\frac{q}{k_{j}}\right)^{-(m_{j}+1)}+\textrm{ (analytic part) } .
        \]
        Since $(k_{j},k_{t})=1$ for $t\neq j$, it follows from Taylor's theorem that for $t\neq j$ the expansion 
        \[
        \cot^{(m_{t})}(\pi k_{t}z) \ = \ \sum_{l_{t}=0}^{\infty}\frac{(\pi k_{t})^{l_{t}}}{l_{t}!}\cot^{(m_{t}+l_{t})}\left(\frac{\pi
k_{t}q}{k_{j}}\right)\left(z-\frac{q}{k_{j}}\right)^{l_{t}}
        \]
        is valid near $\frac{q}{k_{j}}$. Taylor's theorem also yields that
        \[
        \zeta^{(m_{0})}(a,z) \ = \ \sum_{l_{0}=0}^{\infty}\frac{\zeta^{(m_{0}+l_{0})}\left(a,\frac{q}{k_{j}}\right)}{l_{0}!}\left(z-\frac{q}{k_{j}}\right)^{l_{0}}
        \]
        near $\frac{q}{k_{j}}$. Hence, we may write $\res_{z=\frac{q}{k_{j}}}f(z)$
        as 
        \[
        \frac{(-1)^{m_{j}}m_{j}!}{\pi k_{j}}\sum_{{l_{0}+\cdots+\widehat{l_{j}} +\cdots+l_{d}=m_{j}\atop
l_{0},\ldots,\widehat{l_{j}},\ldots,l_{d}\geq0}}\frac{\zeta^{(m_{0}+l_{0})}\left(a,\frac{q}{k_{j}}\right)}{l_{0}!}\prod_{{t=1\atop t\neq j}}^{d}\frac{(\pi
k_{t})^{l_{t}}}{l_{t}!}\cot^{(m_{t}+l_{t})}\left(\frac{\pi k_{t}q}{k_{j}}\right).
        \]      
        Therefore $\sum_{q=1}^{k_{j}-1}\res_{z=\frac{q}{k_{j}}}f(z)$ is given by
\begin{multline*}
\frac{(-1)^{m_{j}}}{\pi}\sum_{{l_{0}+\cdots+\widehat{l_{j}}+\cdots+l_{d}=m_{j}\atop l_{0},\ldots,\widehat{l_{j}},\ldots,\,l_{d}\geq0}}{m_{j} \choose
l_{0},\ldots,\widehat{l_{j}},\ldots,l_{d}}\left(\prod_{{t=1\atop t\neq j}}^{d}(\pi k_{t})^{l_{t}}\right)\\
\times\frac{1}{k_j}\sum_{q=1}^{k_{j}-1}\zeta^{(m_{0}+l_{0})}\left(a,\frac{q}{k_{j}}\right)\prod_{{t=1\atop t\neq j}}^{d}\cot^{(m_{t}+l_{t})}\left(\frac{\pi
k_{t}q}{k_{j}}\right)\\
=\frac{(-1)^{m_{j}}}{\pi}\sum_{{l_{0}+\cdots+\widehat{l_{j}}+\cdots+l_{d}=m_{j}\atop l_{0},\ldots,\widehat{l_{j}},\ldots,l_{d}\geq0}}{m_{j} \choose
l_{0},\ldots,\widehat{l_{j}},\ldots,l_{d}}\left(\prod_{{t=1\atop t\neq j}}^{d}(\pi k_{t})^{l_{t}}\right)k_{j}^{a-1}c_{-a}(j) \, .
\end{multline*}
        Given that this holds for all $j$, we conclude that 
        \begin{align}
        &\sum_{z_{0}\in E\backslash\{1\}}\res_{\{z=z_{0}\}}f(z) \ = \label{eq:sumResRecip} \\
        &\qquad \sum_{j=1}^{d}\frac{(-1)^{m_{j}}}{\pi}\sum_{{l_{0}+\cdots+\widehat{l_{j}}+\cdots+l_{d}=m_{j}\atop l_{0},\ldots,\widehat{l_{j}},\ldots,l_{d}\geq0}}{m_{j} \choose l_{0},\ldots,\widehat{l_{j}},\ldots,l_{d}}\left(\prod_{{t=1\atop t\neq j}}^{d}(\pi
k_{t})^{l_{t}}\right)k_{j}^{a-1}c_{-a}(j) \, . \nonumber
        \end{align}
        We now compute $\res_{z=1}f(z)$. For each $j\in\{1,2,\dots,d\}$, we know that $\cot(\pi k_{j}z)$ has an expansion about 1 of the form
        \[
        \cot(\pi k_{j}z)=\frac{1}{\pi k_{j}(z-1)}+\sum_{n=0}^{\infty}\frac{(2i)^{n+1}B_{n+1}(\pi k_{j})^{n}}{(n+1)!}(z-1)^{n},
        \]
        so the Laurent expansion of $\cot^{(m_{j})}(\pi k_{j}z)$ about 1 is given by
        \begin{align*}
        \cot^{(m_{j})}(\pi k_{j}z) & =\frac{(-1)^{m_{j}}m_{j}!}{\pi k_{j}(z-1)^{m_{j}+1}}+\sum_{l_{j}=0}^{\infty}\frac{(2i)^{l_{j}+m_{j}+1}B_{l_{j}+m_{j}+1}(\pi
k_{j})^{l_{j}+m_{j}}(l_{j}+1)^{(m_{j})}}{(l_{j}+m_{j}+1)!}(z-1)^{l_{j}}\\
        & =\sum_{l_{j}=-\infty}^{\infty}a_{l_{j}}(z-1)^{l_{j}},
        \end{align*}
        where
        \[
        a_{l_{j}}=\begin{cases}
        \frac{(2i)^{l_{j}+m_{j}+1}B_{l_{j}+m_{j}+1}(\pi k_{j})^{l_{j}+m_{j}}(l_{j}+1)^{(m_{j})}}{(l_{j}+m_{j}+1)!} & \textnormal{if }l_{j}\geq0,\\
        \frac{(-1)^{m_{j}}m_{j}!}{\pi k_{j}} & \textnormal{if }l_{j}=-(m_{j}+1),\\
        0 & \textnormal{otherwise}.
        \end{cases}
        \]
        Now, Taylor's theorem implies that the expansion of $\zeta^{(m_{0})}(a,z)$ about 1 is of the form
        \[
        \zeta^{(m_{0})}(a,z)=\sum_{l_{0}=0}^{\infty}a_{l_{0}}(z-1)^{l_{0}},
        \]
        where
        \[
        a_{l_{0}}=\frac{\zeta^{(m_{0}+l_{0})}(a,1)}{l_{0}!}=\frac{(-1)^{m_{0}+l_{0}}a^{(m_{0}+l_{0})}\zeta(a+m_{0}+l_{0})}{l_{0}!}.
        \]
        Therefore, 
        \[
        \res_{z=1}f(z)=\sum_{l_{0}=0}^{m_{1}+\cdots+m_{d}+d-1}\sum_{l_{1}+\cdots+l_{d}=-l_{0}-1}\prod_{j=0}^{d}a_{l_{j}}.
        \]
        Since $\frac{\partial}{\partial z}\,\zeta(a,z)=-a\,\zeta(a+1,z)$, 
        \begin{align*}
                \zeta^{(m_{0})}(a,z+1) \ &= \ (-1)^{m_{0}}\zeta(a+m_{0},z+1)a^{(m_{0})}
                \ = \ (-1)^{m_{0}}a^{(m_{0})}\sum_{n=0}^{\infty}\frac{1}{(n+z+1)^{a+m_{0}}}\\
                 &= \ (-1)^{m_{0}}a^{(m_{0})}\left(\zeta(a+m_{0},z)-\frac{1}{z^{a+m_{0}}}\right)
                \ = \ \zeta^{(m_{0})}(a,z)-\frac{(-1)^{m_{0}}a^{(m_{0})}}{z^{a+m_{0}}} \, .
        \end{align*}
        This means that
        \[
        \int_{1+\epsilon-iM}^{1+\epsilon+iM}f(z) \, dz \, +\int_{\epsilon+iM}^{\epsilon-iM}f(z) \, dz \ = \
(-1)^{m_{0}}a^{(m_{0})}\int_{\epsilon+iM}^{\epsilon-iM}\frac{\prod_{j=1}^{d}\cot^{(m_{j})}(\pi k_{j}z)}{z^{a+m_{0}}} \, dz \, .
        \]
        As in the proof of Theorem \ref{thm:GeneralReciprocity}, it follows
        from Lemmas \ref{lem:AsympCot} and \ref{lem:AsympHurwitz} that the
        integrals along the horizontal segments of $C(M,\epsilon)$ vanish,
        so the Cauchy residue theorem implies that 
        \[
        \sum_{z_{0}\in E}\res_{z=z_{0}}f(z) \ = \ \frac{(-1)^{m_{0}}a^{(m_{0})}}{2\pi i}\int_{\epsilon+i\infty}^{\epsilon-i\infty}\frac{\prod_{j=1}^{d}\cot^{(m_{j})}(\pi
k_{j}z)}{z^{a+m_{0}}} \, dz \, .
        \]
        The result then follows from (\ref{eq:sumResRecip}) and the computation
        of $\res_{z=1}f(z)$.
\end{proof}

An analogue Theorem \ref{thm:nReciprocity} is valid for generalized Bettin--Conrey sums.

\begin{thm}
        \label{thm:PartRecip}Let $d\geq2$ and suppose that $k_{1},\dots,k_{d}$
        is a list of pairwise coprime positive integers and $m_{0},m_{1},\dots,m_{d}$
        are nonnegative integers. If $n>1$ is an integer and $m_{0}+n+d+\sum_{j=1}^{d}m_{j}$
        is odd, then 
\begin{align*}
&\sum_{j=1}^{d}\frac{(-1)^{m_{j}}}{\pi}\sum_{{l_{0}+\cdots+\widehat{l_{j}}+\cdots+l_{d}=m_{j}\atop l_{0},\ldots,\widehat{l_{j}},\ldots,l_{d}\geq0}}{m_{j} \choose
l_{0},\ldots,\widehat{l_{j}},\ldots,l_{d}}\left(\prod_{{t=1\atop t\neq j}}^{d}(\pi k_{t})^{l_{t}}\right)k_{j}^{n-1}c_{-n}(j) \ = \\
&-\left(\sum_{l_{0}=0}^{m_{1}+\cdots+m_{d}+d-1}\sum_{l_{1}+\cdots+l_{d}=-l_{0}-1}\prod_{j=0}^{d}a_{l_{j}}\right)+\frac{(-1)^{m_{0}+1}n^{(m_{0})}}{2}\sum_{l_{1}+\cdots+l_{d}=n+m_{0}-1}\prod_{j=1}^{d}a_{l_{j}}
\end{align*}
        where 
        \[
        c_{-n}(j) \ = \ c_{-n}\left(\begin{array}{c|ccccc}
        k_{j} & k_{1} & \cdots & \widehat{k_{j}} & \cdots & k_{d}\\
        m_{0}+l_{0} & m_{1}+l_{1} & \cdots & \widehat{m_{j}+l_{j}} & \cdots & m_{d}+l_{d}
        \end{array}\right)
        \]
        and
\[
a_{l_{j}}=\begin{cases}
\frac{(-1)^{m_{0}+l_{0}}n^{(m_{0}+l_{0})}\zeta(n+m_{0}+l_{0})}{l_{0}!} & \textnormal{if }j=0\textnormal{ and }l_{0}\geq0,\\
\frac{(2i)^{l_{j}+m_{j}+1}B_{l_{j}+m_{j}+1}(\pi k_{j})^{l_{j}+m_{j}}(l_{j}+1)^{(m_{j})}}{(l_{j}+m_{j}+1)!} & \textnormal{if }j\neq0\textnormal{ and }l_{j}\geq0,\\
\frac{(-1)^{m_{j}}m_{j}!}{\pi k_{j}} & \textnormal{if }j\neq0\textnormal{ and }l_{j}=-(m_{j}+1),\\
0 & \textnormal{otherwise.}
\end{cases}
\]
\end{thm}
\begin{proof}
        As in the proof of Theorem \ref{thm:nReciprocity}, the contour $\widetilde{C}(M,\epsilon)$
        is defined as the positively oriented rectangle with vertices $1+iM$,
        $iM$, $-iM$ and $1-iM$, with indentations (to the right) of radius
        $0<\epsilon<\min_{1\leq j\leq d}\left\{ \frac{1}{k_{j}}\right\} $
        around 0 and 1 (see Figure~\ref{fig:IndentedPath}). Since this closed contour contains the same poles
        of $f$ as $C(M,\epsilon)$, we may apply Cauchy's residue theorem,
        letting $M\rightarrow\infty$, and we only need to determine
$\lim_{M\rightarrow\infty}\int_{\widetilde{C}(M,\epsilon)}f(z) \, dz$
        in order to deduce a reciprocity law for the generalized Bettin--Conrey sums of the form~$c_{-n}$.
        
        Given that $m_{0}+n+d+\sum_{j=1}^{d}m_{j}$ is odd, the function 
        \[
        g(z) \ = \ \frac{(-1)^{m_{0}}n^{(m_{0})}\prod_{j=1}^{d}\cot^{(m_{j})}(\pi k_{j}z)}{z^{n+m_{0}}}
        \]
        is odd:
        \[
        g(-z)
        \ = \ \frac{(-1)^{m_{0}+d+\sum_{j=1}^{d}m_{j}}n^{(m_{0})}}{(-1)^{n+m_{0}}z^{n+m_{0}}}\prod_{j=1}^{d}\cot^{(m_{j})}(\pi k_{j}z)
        \ = \ \frac{(-1)^{d+\sum_{j=1}^{d}m_{j}}}{(-1)^{n+m_{0}}}g(z)
        \ = \ -g(z) \, .
        \]
        Let $\gamma(M,\epsilon)$ be the indentation around zero along $\widetilde{C}(M,\epsilon)$, then
        \begin{align*}
        \lim_{M\rightarrow\infty}\int_{\widetilde{C}(M,\epsilon)}f(z) \, dz
        \ &= \ \lim_{M\rightarrow\infty}\left(\int_{iM}^{\epsilon i}g(z) \, dz+\int_{-\epsilon i}^{-iM}g(z) \, dz\right)+\int_{\gamma(M,\epsilon)}g(z) \, dz\\
        &= \ \int_{\gamma(M,\epsilon)}g(z) \, dz \, .
        \end{align*}
        Given that $g$ is odd, the Cauchy residue theorem implies that
        \[
        \int_{\gamma(M,\epsilon)}g(z) \, dz \ = \ -\pi i\res_{z=0}g(z)
        \]
        and it follows that 
        \[
        \sum_{z_{0}\in E}\res_{z=z_{0}}f(z) \ = \ -\frac{1}{2}\res_{z=0}g(z) \, .
        \]
        For each $j\in\{1,2,\dots,d\}$ we have an expansion of the form
\[
\cot^{(m_{j})}(\pi k_{j}z)=\sum_{l_{j}=-\infty}^{\infty}a_{l_{j}}z^{l_{j}},
\]
where 
\[
a_{l_{j}}=\begin{cases}
\frac{(2i)^{l_{j}+m_{j}+1}B_{l_{j}+m_{j}+1}(\pi k_{j})^{l_{j}+m_{j}}(l_{j}+1)^{(m_{j})}}{(l_{j}+m_{j}+1)!} & \textnormal{if }l_{j}\geq0,\\
\frac{(-1)^{m_{j}}m_{j}!}{\pi k_{j}} & \textnormal{if }l_{j}=-(m_{j}+1),\\
0 & \textnormal{otherwise}.
\end{cases}
\]
Thus $\res_{z=0}g(z)$ is given by 
\[
(-1)^{m_{0}}n^{(m_{0})}\sum_{l_{1}+\cdots+l_{d}=n+m_{0}-1}\prod_{j=1}^{d}a_{l_{j}} \, .
\]
Therefore,
\[
\sum_{z_{0}\in E}\res_{z=z_{0}}f(z) \ = \ \frac{(-1)^{m_{0}+1}n^{(m_{0})}}{2}\sum_{l_{1}+\cdots+l_{d}=n+m_{0}-1}\prod_{j=1}^{d}a_{l_{j}} \, ,
\]
        which concludes our proof.
\end{proof}

From Theorems \ref{thm:GeneralRecip} and \ref{thm:PartRecip}, we deduce a computation of the integral
\[ 
\int_{\epsilon+i\infty}^{\epsilon-i\infty}\frac{\prod_{j=1}^{d}\cot^{(m_{j})}(\pi k_{j}z)}{z^{n+m_{0}}} \, dz
\]
in terms of the sequences $B(m_{j})_{l_{j}}$, whenever $n\in\mathbb{Z}_{>1}$
and $m_{0}+n+d+\sum_{j=1}^{d}m_{j}$ is odd, which generalizes Corollary~\ref{cor:nIntegral}:

\begin{cor}
        \label{cor:CollapsInt}Let $d\geq2$ and suppose that $k_{1},\dots,k_{d}$
        is a list of pairwise coprime positive integers and $m_{0},m_{1},\dots,m_{d}$
        are nonnegative integers. If $n>1$ is an integer and 
        $
        m_{0}+n+d+\sum_{j=1}^{d}m_{j}
        $
        is odd, then for all $0<\epsilon<\min_{1\leq j\leq d}\left\{ \frac{1}{k_{j}}\right\} $,
\[
\int_{\epsilon+i\infty}^{\epsilon-i\infty}\frac{\prod_{j=1}^{d}\cot^{(m_{j})}(\pi k_{j}z)}{z^{n+m_{0}}}dz=-\pi i\sum_{l_{1}+\cdots+l_{d}=n+m_{0}-1}\prod_{j=1}^{d}a_{l_{j}} \, ,
\]
where the sequences $\{a_{l_{j}}\}$ are as in Theorem \ref{thm:PartRecip}, for $j=1,2,\ldots,d$.
        
\end{cor}

The consideration of the case $m_{0}=m_{1}=\cdots=m_{n}=0$ leads to the definition of \textit{higher\--dimensional Bettin--Conrey sums},
\[
c_{a}(k_{0};k_{1},\dots,k_{n}) \ = \ c_{a}\left(\begin{array}{c|ccc}
k_{0} & k_{1} & \cdots & k_{n}\\
0 & 0 & \cdots & 0
\end{array}\right) \ = \ k_{0}^{a}\sum_{m=1}^{k_{0}-1}\zeta\left(-a,\frac{m}{k_{0}}\right)\prod_{l=1}^{n}\cot\left(\frac{\pi k_{l}m}{k_{0}}\right),
\]
for $a\neq-1$ complex and $k_{0},k_{1},\dots,k_{n}$ a list of positive
numbers such that $(k_{0},k_{j})=1$ for each $j\neq0$. 

Of course, higher\--dimensional Bettin--Conrey sums satisfy
Theorems \ref{thm:GeneralRecip} and \ref{thm:PartRecip}. In particular, if $0<\epsilon<\min_{1\leq l\leq d}\left\{ \frac{1}{k_{l}}\right\} $, $\Re(a)>1$ and $k_{1},\dots,k_{d}$
is a list of pairwise coprime positive integers, then 
\begin{align*}
&\sum_{j=1}^{d}k_{j}^{a-1}c_{-a}(k_{j};k_{1},\ldots,\widehat{k_{j}},\ldots,k_{d}) \ = \\
&\qquad -\pi\sum_{l_{0}=0}^{d-1}\sum_{l_{1}+\cdots+l_{d}=-l_{0}-1} a_{ l_0 } a_{ l_1 } \cdots a_{ l_d } +
\frac{1}{2i}\int_{\epsilon+i\infty}^{\epsilon-i\infty}\frac{\prod_{j=1}^{d}\cot(\pi k_{j}z)}{z^{a}} \, dz \, ,
\end{align*}
where 
\[
a_{l_{j}}=\begin{cases}
\frac{(-1)^{l_{0}}a^{(l_{0})}\zeta(a+l_{0})}{l_{0}!} & \textnormal{if }j=0\textnormal{ and }l_{0}\geq0,\\
\frac{(2i)^{l_{j}+1}B_{l_{j}+1}(\pi k_{j})^{l_{j}}}{(l_{j}+1)!} & \textnormal{if }j\neq0\textnormal{ and }l_{j}\geq0,\\
\frac{1}{\pi_{k_{j}}} & \textnormal{if }j\neq0\textnormal{ and }l_{j}=-1,\\
0 & \textnormal{otherwise}.
\end{cases}
\]

        
\section{Derivative Cotangent Sums and Critical Values of Estermann Zeta}\label{sec:estermann}

As usual, for $a, x\in\C$, let $\sigma_a(n)=\sum_{d|n}d^a$ and $e(x)=e^{2\pi i x}$. For a given rational number $x$, the \emph{Estermann zeta function} is defined through the
Dirichlet series
\begin{eqnarray}\label{Estermann}
E(s,x,a) \ = \ \sum_{n\geq 1}\sigma_a(n) \, e(nx) \, n^{-s},
\end{eqnarray}
initially defined for $\Re(s)>\max\left(1,\Re(a)+1\right)$ and analytically continued to the whole $s$-plane with possible poles at $s=1,\ a+1$.
For $x= \frac p q$ with $(p,q)=1$ and $q>1$, 
\[E(s,x,a)-q^{1+a-2s}\zeta(s-a)\zeta(s)
\]
is an entire function of $s$. 
By use of the Hurwitz zeta function we observe that
\begin{eqnarray}\label{Hurwitz}
E(s,x,a) \ = \ q^{a-2s}\sum_{m,n=1}^q e(mnx) \, \zeta \left( s-a, \tfrac m q \right) \zeta \left( s, \tfrac n q \right) .
\end{eqnarray}
We consider the sums
\begin{eqnarray}\label{Derivative-Cotangent}
C(a,s,x) \  = \ q^a\sum_{m=1}^{q-1} e(mx) \, \Phi(-s,1,e(mx))\, \zeta \left(-a, \tfrac m q \right) ,
\end{eqnarray}
where $\Phi(s,z,\lambda)=\sum_{n\geq 0}\frac{\lambda^n}{(z+n)^s}$ is \emph{Lerch's transcendent function}, defined for 
$z\ne 0,-1,-2,\dots, |\lambda|<1; \Re(s)>1, |\lambda|=1$, and analytically continued in $\lambda$.\\

The purpose of this section is to establish relationships between $C(a,s,x)$ and values of the Estermann zeta function at integers~$s$. 
We start with some preliminary results.
\begin{lem}\label{difference} Let $k$ be a nonnegative integer. 
Then
\begin{align}
  \lambda\Phi(s,z+1,\lambda) \ &= \ \Phi(s,z,\lambda)-z^{-s} \label{formula1}\\
  \Phi(-k,z,\lambda) \ &= \ -\frac{B_{k+1}(z;\lambda)}{k+1} \label{formula2}\\
  B_{k}(0;e(x)) \ &= \ 
 \begin{cases}
\frac{1}{2i}\cot(\pi x)-\frac{1}2&  \textnormal{if \ensuremath{k=1},}\\
\frac{k}{(2i)^{k}}\cot^{(k-1)}(\pi x)& \textnormal{if \ensuremath{k>1}.}
\end{cases}\label{formula3}
\end{align}
\end{lem}
\begin{proof}
Equation \eqref{formula1} follows from the special case $m=1$ in \cite[(25.14.4)]{MR2723248}.
Equation \eqref{formula2} can be found in  \cite[p.164]{Apostol3}.
Equation \eqref{formula3} follows from \cite[Lemma 2.1]{Adamchik} and \cite[Theorem~4]{BC}.
\end{proof}

Lemma \ref{difference} implies that for a positive integer $s=k$, the sum $C(a,k,x)$ defined in \eqref{Derivative-Cotangent} is, up to a constant factor, the $k\th$-derivative
cotangent sum

\[
C(a,k,x) \ = \ -\frac{1}{(2i)^{k+1}}\  q^a\sum_{m=1}^{q-1}\cot^{(k)}(\pi mx)\, \zeta \left(-a, \tfrac m q \right) .
\]

\begin{lem}\label{distribution}
Let $p, q$ be coprime positive integers and $x=\frac{p}{q}$.  
For any $n \in \ZZ$ and $z \in \CC$ with $\Re(z)>0$, 
\begin{eqnarray}
\sum_{m=0}^{q-1}e(mnx) \, \zeta \left( s, z + \frac m q \right) \ = \ q^s \, \Phi(s,qz,e(nx)) \, . \label{distribution1} 
\end{eqnarray}
\end{lem}
\begin{proof} Writing $m=kq+j$ with $j=0,\dots, q-1$, we have 
\[
  q^s \, \Phi(s,qz,e(nx))
  \ = \ q^s\sum_{m=0}^{\infty}\frac{e(nmx)}{(m+qz)^s} 
  \ = \ \sum_{j=0}^{q-1}e(njx)\sum_{k=0}^{\infty}\frac1{(k+z+ \frac j q)^s} \, . \qedhere 
\]
\end{proof}
\begin{prop}\label{duals} Let $p, q$ be coprime positive integers and $x= \frac{p}{q}$. Then 
\begin{align*}
E(-s,x,a-s) \ &= \ q^{a}\sum_{m=1}^{q-1} \, e(mx) \, \zeta \left( -a, \tfrac m q \right) \Phi(-s,1,e(mx))+q^{a}\zeta(-s) \, \zeta(-a) \\
E(-s,x,a-s) \ &= \ q^{s}\sum_{n=1}^{q-1}e(nx) \, \zeta \left( -s, \tfrac n q \right) \Phi(-a,1,e(nx))+q^{s}\zeta(-a) \, \zeta(-s) \, .
\end{align*}
\end{prop}


\begin{thm}
Let $a,k$ be nonnegative integers. Then 
\begin{align*}
E\left(-k,x, a-k\right) \ &= \ C(a,k,x)+q^{a} \zeta(-k)\, \zeta(-a) \qquad \textrm{ if } k\geq 1, \\
E\left(-k,x, a-k\right) \ &= \ C(k,a,x)+q^{k} \zeta(-k)\, \zeta(-a) \qquad \textrm{ if } a\geq 1,
\end{align*}
and 
\begin{align*}
E\left(0,x, a\right) \ &= \ C(a,0,x)-\tfrac12 \zeta(-a) \\ 
E\left(0,x, a\right) \ &= \ C(0,a,x)-\tfrac12 \zeta(-a)\  \textrm{ with  $a\geq 1$.}  
\end{align*}
\end{thm}

\begin{cor}
Let $a,k$ be nonnegative integers.  For any rational number $x\neq 0$,
\[
C(a,k,x)-C(k,a,x) \ = \
 \begin{cases}
 0 & \textnormal{if \ensuremath{k=0}  or \ensuremath{a=0},}\\
 \left(q^a-q^k\right)\zeta(-k)\, \zeta(-a) & \textnormal{otherwise}.
 \end{cases}
\]
\end{cor}

\subsection*{Acknowledgements}
We thank Sandro Bettin and an anonymous referee for valuable comments.
Abdelmejid Bayad was partially supported by the FDIR of the Universit\'e d'Evry Val d'Essonne;
Matthias Beck was partially supported by the US National Science Foundation (DMS-1162638).

\end{document}